\documentclass[a4paper]{amsart}
\usepackage[utf8]{inputenc}
\usepackage[english]{babel}
\usepackage{fancyhdr}
\usepackage{amsmath}
\usepackage{amsfonts,amstext}
\usepackage{amssymb,amsthm,amscd,amsxtra,wasysym,graphicx}
\usepackage{mathrsfs,esint,comment}
\usepackage{enumitem,mathtools,dsfont}
       \usepackage{palatino,mathpazo}
  \usepackage{palatino}
\usepackage{charter}
\usepackage{xcolor}

\def\Bibtex{{\rm B\kern-.05em{\sc i\kern-.025em b}\kern-0.08em T\kern-.1667em\lower.7ex\hbox{E}\kern-.125emX}}
\usepackage{hyperref}
\hypersetup{
	unicode=false,          
	pdftoolbar=true,       
	pdfmenubar=true,       
	pdffitwindow=false,     
	pdfstartview={FitH}, 
	pdftitle={Monge-Amp\`ere equations},   
	pdfauthor={Quang-Tuan Dang},   
	colorlinks=true,  
	linkcolor=purple,         
	citecolor=blue,        
	filecolor=green,      
	urlcolor=blue}          

\usepackage{geometry}
\usepackage{fullpage}
\usepackage{indentfirst}

\usepackage{colonequals}
\usepackage{enumitem}

\numberwithin{equation}{section}

\newtheorem{theorem}{Theorem}[section]
\newtheorem{proposition}[theorem]{Proposition}
\newtheorem{corollary}[theorem]{Corollary}
\newtheorem{lemma}[theorem]{Lemma}
\theoremstyle{definition}
\newtheorem{definition}[theorem]{Definition}
\newtheorem{remark}[theorem]{Remark}

\newtheorem{example}[theorem]{Example}

\newcommand{\Vol}{{\rm Vol}}

\newcommand{\ddc}{dd^c}

\def\C^n{\mathbb C^n}

\numberwithin{equation}{section}

\newcommand{\capK}{\text{Cap}}

\newcommand{\PSH}{{\rm PSH}}

\newcommand{\N}{\mathbb{N}}

\usepackage{microtype}
\begin{document}

\title{Singularities vs non-pluripolar Monge--Amp\`{e}re masses}
\author{Quang-Tuan Dang, Hoang-Son Do, Hoang Hiep Pham}

\address{The Abdus Salam International Centre for Theoretical Physics (ICTP), Str. Costiera, 11, 34151 Trieste, TS, Italy}
\email{qdang@ictp.it}
\address{Vietnam Academy of Science and Technology, Institute of Mathematics, 18 Hoang Quoc Viet Road, Cau Giay, Hanoi, Vietnam}
\email{dhson@math.ac.vn $\&$ phhiep@math.ac.vn}
\date{\today}
\subjclass[2020]{32W20, 32U05, 32Q15}
\keywords{Complex Monge-Amp\`ere operators, Singularity type, Capacity}




\begin{abstract}
 The aim of this paper is to compare singularities of closed positive currents whose non-pluripolar complex Monge--Ampère masses equal. We also provide a short alternative proof for the monotonicity of non-pluripolar complex Monge--Ampère masses, generalizing results of Witt-Nystr\"om, Darvas--Di Nezza--Lu, Lu--Nguy\^en and Vu. 
\end{abstract}

\maketitle
\section{Introduction}

One of the central topics of interest in pluripotential theory is the domain of the definition of the complex Monge--Amp\`ere operator. 
Building on the seminal work of Bedford and Taylor \cite{bedford1976dirichlet,bedford1982new}, the Monge--Amp\`{e}re operator is well-defined on the class of locally bounded plurisubharmonic functions.
This was later extended by Demailly~\cite{demailly1993monge} to plurisubharmonic functions with isolated or compactly supported singularities. Furthermore, Cegrell~\cite{cegrell1998pluricomplex,cegrell2004general} introduced the largest class $\mathcal{E}$ of plurisubharmonic functions, on which the Monge--Amp\`ere operator is well-defined and continuous along decreasing sequences.


 Guedj and Zeriahi~\cite{guedj2007weighted} extended the Bedford–Taylor theory~\cite{bedford1987fine} and Cegrell theorey~\cite{cegrell1998pluricomplex} (originally developed in a local setting) to the geometric setting of compact K\"ahler manifolds, where they defined the so-called non-pluripolar Monge--Amp\`ere operator for unbounded potentials. Their approach was later adapted to the setting of big cohomology classes by Boucksom, Eyssidieux, Guedj, and Zeriahi \cite{boucksom2010monge}. In the latter paper, the authors posed a conjecture regarding the monotonicity of Monge--Ampère masses, which was later resolved by D. Witt-Nystr\"om  \cite{wittnystrom19-monotonicity}. Alternative proofs were subsequently provided in \cite{vu2021relative,lu2022hessian,lu2021capacities}.
In this paper, we generalize this result by showing that the complex Monge--Ampère mass of potentials decreases as its singularity type increases {\em in capacity}; cf. Definition~\ref{def: sing-capa}. 

To state our results, let $(X, \omega)$ be a compact K\"{a}hler manifold of dimension $n$ and let $\theta$ be a smooth closed real(1,1) form. 
We let $\PSH(X, \theta)$ denote the set of $\theta$-psh functions on $X$. Recall that
the cohomology class $\{\theta\}$ is {\em pseudoeffective} if $\PSH(X,\theta)\neq \varnothing$. 
 The cohomology class $\{\theta\}$ is is {\em big} if $\PSH(X,\theta-\varepsilon_0\omega)\neq \varnothing$ for some $\varepsilon_0>0$. 

 If $u$ and $v$ are two $\theta$-psh functions on X, then $u$ is said to be {\em less (resp. more) singular} than $v$ if $u  \geq\, v+O(1)$ (resp. $u  \leq\, v+O(1)$). 
 We say that $u$ has {\em the same singularities} as $v$ if $u$ is less singular than $v$, and $v$ is less singular than $u$. 

For a $\theta$-psh function $u$, one defines its {\em non-pluripolar Monge--Amp\`ere measure} by $\theta_u^n:=\theta_u\wedge\ldots\wedge\theta_u$. Fix $\phi\in \PSH(X,\theta)$.
Let
$\mathcal{E}(X,\theta,\phi)$ denote the set of  $\theta$-ph functions $u$ with full mass relative to $\phi$, i.e., $u$ is more singular than $\phi$ and $\int_X\theta^n_u=\int_X\theta_\phi^n$. When $\phi=V_\theta$ is the "the least singular" element of $\PSH(X,\theta)$, we simply denote by $\mathcal{E}(X,\theta)$.


Extending the result of Guedj--Zeriahi~\cite{guedj2007weighted} and Dinew~\cite{dinew2009uniqueness} for the K\"ahler case, Darvas--Di Nezza--Lu~\cite{darvas2018monotonicity,darvas2021log} 
 showed that the non-pluripolar Monge--Ampère measure uniquely determines the potential within a relative full mass class $\mathcal{E}(X,\theta,\phi)$. 
They also provided a characterization of membership in the class $\mathcal{E}(X,\theta,\phi)$, solely in terms of singularity type. Our first theorem below gives another characterization.
\begin{theorem}\label{thm: thm1}
Let $(X, \omega)$ be a compact K\"{a}hler manifold and $\theta$ a smooth closed real (1, 1)-form on $X$ whose cohomology class is big.
Let $\phi\in\PSH(X,\theta)$. If $\varphi,\psi\in\mathcal{E}(X,\theta,\phi)$, then for each $c>0$, there exists a function $ h\in\mathcal{E}(X,\omega)$ such that $$|\varphi-\psi|\leq -ch.$$
Conversely, if $\varphi,\psi\in\PSH(X,\theta)$ such that $|\varphi-\psi|\leq -ch$ for some $c>0$ and $h\in\mathcal{E}(X,\omega)$, then both $\varphi$ and $\psi$ belong to $\mathcal{E}(X,\theta,\max(\varphi,\psi))$.

In particular, $\varphi\in\mathcal{E}(X,\theta,\phi)$ if and only if $\varphi$ is more singular than $\phi$ and $\varphi\geq \phi+ch $ for some $c>0$ and $h\in\mathcal{E}(X,\omega)$.
\end{theorem}  
The first statement of this theorem is a direct consequence of a result of the first author~\cite{dang2021continuity} (see Proposition~\ref{prop: homega}), whose proof relies on the envelope technique, developed in the series of recent works~\cite{darvas2018singularity,darvas2018monotonicity,darvas2020metric}. 
For the second statement,
 inspired by the idea in \cite{Ahag09-MA-pluripolar}, we establish a slight generalization of the monotonicity of non-pluripolar Monge--Ampère masses, as shown in \cite{wittnystrom19-monotonicity} (see also \cite{vu2021relative,lu2022hessian,lu2021capacities} for different proofs), thereby providing an alternative proof of Witt-Nystr\"om's result.
\begin{theorem}\label{mainthm}
Let $(X, \omega)$ be a compact K\"{a}hler manifold of dimension $n$. Let $\theta_j$, $j\in\{1,\ldots,n\}$, be smooth closed real $(1,1)$-forms on $X$ whose cohomology classes are pseudoeffective. Assume $\varphi_j,\psi_j \in \PSH (X,\theta_j )$. If there exist $h \in \mathcal E (X, \omega )$ and $c\geq 0$ such that 
$$\lim\limits_{ k\to+\infty }k^n \capK_{\omega} ( \{ \varphi_j \ < \psi_j+ c h - k \} ) = 0,$$
for all $1\leq j\leq n$, then 
$$\int_X  \theta_{ 1, \varphi_1 }\wedge ...\wedge \theta_{ n, \varphi_n }  \geq \int_X  \theta_{ 1, \psi_1 }\wedge ...\wedge \theta_{ n, \psi_n }  .$$
\end{theorem}
Here, the notation of Monge--Ampère capacity $\capK_\omega$ was introduced by S. Ko\l odziej \cite{kolodziej2003monge}: for any Borel set $E\subset X$, one defines
$$\capK_{ \omega }(E):=\sup\Big\{\int_X\omega_u^n: u\in\PSH(X,\omega),\, 0\leq u\leq 1 \Big\}.$$



\subsection*{Acknowledgement} We are grateful to Hoang-Chinh Lu for his  comments on
a preliminary draft.  
This work was done while the authors were visiting Vietnam Institute for Advanced Study in Mathematics (VIASM), and we would like to thank VIASM for its hospitality.

\subsection*{Ethics declarations} The authors declare no conflict of interest.
\section{Preliminaries}
In the whole article, we let $X$ denote a compact K\"ahler manifold of dimension $n$, equipped with a K\"ahler form $\omega$, and $\theta$ a smooth closed real $(1,1)$-form on $X$. Without loss of generality, we may assume that $\theta\leq \omega$.
\subsection{Non-pluripolar products}
Recall that a function $u:X\to\mathbb{R}\cup\{-\infty\}$ is quasi-plurisubharmonic
(quasi-psh) if, locally, $u$ can be written as the sum of a plurisubharmonic function and a smooth function. 
We say that a function $u$ is $\theta$-plurisubharmonic ($\theta$-psh) if it is quasi-psh and $\theta_u:=\theta+\ddc u\geq 0$ in the sense of currents.  We let $\PSH(X, \theta)$ denote the set of $\theta$-psh functions on $X$. Recall that
the cohomology class $\{\theta\}$ is {\em pseudoeffective} if $\PSH(X,\theta)\neq \varnothing$. 
 The cohomology class $\{\theta\}$ is is {\em big} if $\PSH(X,\theta-\varepsilon_0\omega)\neq \varnothing$ for some $\varepsilon_0>0$.

Let $x_0\in X$. Fixing a holomorphic chart $x_0\in V_{x_0}\subset X$,  the {\em Lelong number} $\nu(\varphi,x_0)$ of a quasi-psh function $\varphi$ at  $x_0\in X$ is defined as follows:
		\begin{align*}
		\nu(\varphi,x_0):=\sup\{\gamma\geq 0: \varphi(z)\leq \gamma\log\|z-x_0\|+O(1), \; \text{on}\; V_{x_0}\}.  
		\end{align*}
 Remark that this definition does not depend on the choice of local holomorphic charts.


\smallskip

Assume $u,v\in \PSH(X,\theta)$. We say that $u$ is {\it less (resp. more) singular} than $v$, and denote by $u\succeq v$ (resp. $u \preceq v$), if there exists a constant $C$ such that $u+C\geq v$ (resp. $u\leq v+C$) on $X$. We say that $u$, $v$ have the {\em same singularity type}, and denote by $u\simeq v$ if $u \preceq v$ and $u\succeq v$.
A $\theta$-psh function $u$ is said to have {\em minimal singularities} if it is less singular than any other $\theta$-psh function.
We introduce the extremal function $V_\theta$ defined by
\[ V_\theta(x):=\sup\{\varphi(x)\in\PSH(X,\theta): \varphi\leq 0 \}.\]
One can see that $V_\theta$ is a $\theta$-psh function with minimal singularities. 
In particular, if $\theta$ is semi-positive, then $V_\theta=0$.

Let $\theta_1,\ldots,\theta_p$ be smooth closed real (1,1)-forms for $1\leq p\le n$. 
Let $u_j$ be $\theta_j$-psh functions for $j\in\{1,\ldots,p\}$ and put $\theta_{j,u_j}:=\theta_j +\ddc u_j$. We recall how to define the non-pluripolar product $\theta_{u_1} \wedge \cdots \wedge \theta_{u_n}$. We write locally $\theta_{j,u_j}= \ddc v_j$, where $v_j$ is psh. By \cite{bedford1987fine, boucksom2010monge}, one knows that the sequence of positive currents 
$$\mathbf{1}_{\cap_{j=1}^p \{v_j>-k\}}\ddc \max\{v_1, -k\} \wedge \cdots \wedge \ddc \max\{v_p, -k\}$$ 
is increasing in $k \in \N$ and converges to a closed positive current, which is independent of the choice of local potentials $v_j$'s. Thus, we obtain a well-defined closed positive current on $X$ which is called the \emph{non-pluripolar product} $\theta_{1,u_1} \wedge \cdots \wedge \theta_{p,u_p}$ of $\theta_{1,u_1}, \ldots, \theta_{p,u_p}$.  
For any $u\in\PSH(X,\theta)$, the {\em non-pluripolar complex Monge--Amp\`ere measure} of $u$ is given by \[\theta_u^n:= (\theta+\ddc u)^n.\]
Given a potential $\phi\in\PSH(X,\theta)$, we let $\PSH(X,\theta,\phi)$ denote the set of $\theta$-psh functions $u$ such that $u\leq\phi$. We denote by $\mathcal{E}(X,\theta,\phi)$ the set of functions $u\in\PSH(X,\theta,\phi)$ with {\em full Monge--Amp\`ere mass relative} to $\phi$, i.e.,
$\int_X\theta_u^n=\int_X\theta_\phi^n$.
If $\phi=V_\theta$ we simply denote by $\mathcal{E}(X,\theta)$. The volume of $\theta$ is defined by $\Vol(\theta):=\int_X\theta_{V_\theta}^n.$
We remark that by~\cite[Theorem 4.7]{boucksom2002volume}, the class $\{\theta\}$ is big if and only if $\Vol(\theta)>0$


	We recall here the plurifine locality of the non-pluripolar Monge--Amp\`ere product (see~\cite[Corollary 4.3]{bedford1987fine} or \cite[Section 1.2]{boucksom2010monge}) for later use. This topology is the coarsest such that all quasi-psh functions with values in $\mathbb{R}$ are continuous.
\begin{lemma}\label{lem: plurifine}
	Assume that $\varphi$, $\psi$ are $\theta$-psh functions such that $\varphi=\psi$ on an open set $U$ in the plurifine topology. Then 
	\begin{equation*}
	\mathbf{1}_U\theta_\varphi^n=\mathbf{1}_U\theta_\psi^n.
	\end{equation*}
\end{lemma}
For practice, we stress that sets of the form
$\{u < v\}$, where $u$ and $v$ are quasi-psh functions, are open in the plurifine topology. Lemma \ref{lem: plurifine} will be referred to as the {\em plurifine locality property}.

\begin{lemma}\label{lem: maxprin}
	Let $\varphi,\psi\in\PSH(X,\theta)$. Then
	\[\theta_{\max(\varphi,\psi)}^n\geq \mathbf{1}_{\{\psi\leq\varphi \}}\theta_\varphi^n+\mathbf{1}_{\{\varphi<\psi\}}\theta_{\psi}^n. \]
    In particular, if $\varphi\leq\psi$ then $\mathbf{1}_{\{\varphi=\psi\}}\theta^n_\varphi\leq \mathbf{1}_{\{\varphi=\psi\}}\theta^n_\psi$.
\end{lemma}

\subsection{Quasi-envelopes}
Given a measurable function $f:X\to\mathbb{R}$, we define the {\em $\theta$-psh envelope} of $h$ by
\begin{equation*}
P_\theta(f):=(\sup\{u\in\PSH(X,\theta): u\leq f\;\text{on}\, X \})^*,
\end{equation*} where the star means we take the upper semi-continuous regularization. 

Given a $\theta$-psh function $\phi$, 
Ross and Witt Nystr\"om~\cite{ross2014analytic} introduced the "rooftop envelope" as follows
\[P_\theta[\phi](f)= \left( \lim_{C\to +\infty}P_\theta(\min(\phi+C,f))\right)^*. \]
When $f=0$ we simply write $P_\theta[\phi]$. 
\begin{definition}
A function $\phi\in\PSH(X,\theta)$ is called a {\em model potential} if $\int_X\theta_\phi^n>0$ and $\phi=P_\theta[\phi]$.
\end{definition}
We recall the first technical result about quasi-envelopes. 
\begin{theorem} \label{thm: envelope}
	Assume that $f$ is a quasi-continuous function and $P_\theta(f)\in\PSH(X,\theta)$. Then, $P_\theta(f)\leq f$ outside a pluripolar set, and $\theta_{P_\theta(f)}^n$ is concentrated
	on the contact set $\{P_\theta(f)=f \}$.
\end{theorem}
\begin{proof}
 See \cite[Theorem 2.2]{darvas2023relative}.
\end{proof}

\begin{proposition}\label{prop: homega}
Let $\{\theta\}$ be a big cohomology class and $\phi\in\PSH(X,\theta)$. Let $\varphi\in\mathcal{E}(X,\theta,\phi)$ and $\psi\in \PSH(X,\theta,\phi)$. Then for any $b>0$, $P_\omega(b\varphi-b\psi)\in\mathcal{E}(X,\omega)$.
\end{proposition}
\begin{proof}
The proof was given in~\cite[Proposition 2.10]{dang2021continuity} in the special case when $\phi=V_\theta=\psi$. The proof is based on Theorem \ref{thm: envelope}, the plurifine locality (Lemma \ref{lem: plurifine}), and the monotonicity of Monge--Amp\`ere masses (Theorem~\ref{mainthm}).
The general case, where $\phi$ is an arbitrarily $\theta$-psh function with positive mass, follows by the same argument.
\end{proof}
As a consequence, we confirm the question posed by Berman--Boucksom--Eyssidieux--Guedj--Zeriahi~\cite[Question 36]{dinew2016open}, which was previously resolved by Darvas--Di Nezza--Lu~\cite{darvas2018singularity}.
\begin{corollary}
Assume that $\{\theta\}$ is a big and nef cohomology class and $\varphi\in\mathcal{E}(X,\theta)$. Then, $\varphi$ has zero Lelong number at all points. 
\end{corollary} 
\begin{proof}
Let $\varphi\in \mathcal{E}(X,\theta)$. 
We apply Proposition~\ref{prop: homega} with $b=1$, $\phi=\psi=V_\theta$ to obtain $h\colonequals P_\omega(\varphi-V_\theta)\in\mathcal{E}(X,\omega)$. By Theorem~\ref{thm: envelope}, we have $V_\theta+h\leq\varphi$ almost everywhere, and hence everywhere on $X$. It follows that $\nu(\varphi,x)=\nu(V_\theta,x)$ for all $x\in X$. From~\cite[Propositions 3.2, 3.6]{boucksom2004divisorial}, we have $\nu(V_\theta,x)=0$ for all $x\in X$, which completes the proof. 
\end{proof}
We also point out a slight generalization of Proposition~\ref{prop: homega}.
\begin{proposition}\label{prop: hOmega}
	\label{prop_imp}
	Let $\{\theta\}$ be a big cohomology class. Let $\varphi,\psi\in\PSH(X,\theta)$ be such that $P_\theta[\varphi]$ is less singular than $\psi$. Then for any $b>0$, $P_\omega(b\varphi-b\psi)\in\mathcal{E}(X,\omega)$.  
\end{proposition}
\begin{proof}
	We refer the reader to \cite[Proposition 2.4]{dang2023continuity}. 
\end{proof}
\begin{definition}\label{def: sing-capa}
Assume $\varphi,\psi\in\PSH(X,\theta)$. We say that $\varphi$ is {\em less singular in capacity} than $\psi$, denoted by $\varphi\succeq_C \psi$, if there exist $h\in\mathcal{E}(X,\omega)$ and $c\geq 0$ such that
\begin{equation}\label{eq: def-sing-capa} \lim_{k\to+\infty}k^n\capK_{\omega}(\{\varphi<\psi+ch-k\})=0.\end{equation}
We say $\varphi$, $\psi$ have the {\em same singularity
type in capacity} if $\varphi\succeq_C \psi$ and $\psi\succeq_C \varphi$.
\end{definition}
 
\section{Proof of main theorems}

First, we prove the following proposition, which states that potentials with the same singularity type also have the same global mass. 
This is a conjecture of Boucksom--Eyssidieux--Guedj-Zeriahi~\cite[page 217]{boucksom2010monge}, which was initially proved by Witt--Nystr\"om~\cite{wittnystrom19-monotonicity}.

\begin{proposition}\label{prop1}
Let $\varphi,\psi\in\PSH(X,\theta)$. If $\varphi$ and $\psi$ have the same singularity type, then $\int_X\theta_\varphi^n=\int_X\theta_\psi^n$.  
\end{proposition}
We would like to mention that several alternative proofs of this conjecture were later given by Lu--Nguy\^en~\cite{lu2022hessian} using the monotonicity of the energy functional, Lu~\cite{lu2021capacities} using a standard approximation process, and Vu~\cite{vu2021relative} where relative non-pluripolar products of
currents are constructed. We provide below a short alternative proof.
\begin{proof} First we  can assume that the class $\{\theta\}$ is big. Indeed, if this is not the case we can just replace $\theta$ with $\theta+\varepsilon\omega$. Using the multi-linearity
of the non-pluripolar product (\cite[Proposition 1.4]{boucksom2010monge}), we can let $\varepsilon\to 0$ to conclude. 

We may assume that $0\geq \varphi,\psi$. 
We divide the proof into two steps.

\noindent {\bf Step 1:} We assume that $\theta$ is semi-positive.

We argue the same as in the proof of \cite[Theorem 4.1]{Ahag09-MA-pluripolar}. Fix $\varepsilon>0$ sufficiently small.
We set
$$\varphi_k = \max (\varphi, - k - 1 ),\quad u_k = (1 - 2\varepsilon ) \varphi_k - 2\varepsilon k,\; \text{and}\quad v_k = \max ( u_k , \psi - \varepsilon k ).$$
We set $E_k := \{ \varphi < - k \} \cap \{ \varphi > \psi -\varepsilon k \}$. We observe that  $\varphi_k+k<0$ on $E_k$, it follows that
$$(1 - 2\varepsilon ) \varphi_k - 2\varepsilon k=\varphi_k-2\varepsilon(\varphi_k+k)>\psi-\varepsilon k $$ there, so $v_k=(1 - 2\varepsilon ) \varphi_k - 2\varepsilon k$ on $E_k$. By plurifine locality (Lemma \ref{lem: plurifine}), we have
$$\theta_{v_k}^n = (\theta + (1-2\varepsilon )dd^c \varphi_k )^n = ((1 -  2\varepsilon ) \theta_{ \varphi_k } + 2\varepsilon\theta )^n$$
on  $E_k$. It follows that
$$\int_{ \{u_k> \psi - \varepsilon k  \} } \theta_{ v_k }^n  \geq ( 1 - 2\varepsilon )^n \int_{ E_k }  \theta_{ \varphi_k }^n.$$
On the other hand, Stoke's theorem yields
$\int_{ X } \theta_{ v_k }^n = \int_{ X } \theta_{ \varphi_k }^n = \int_{ X } \theta^n.$ Observe that $v_k=\psi-\varepsilon k$ on $\{u_k\leq \psi-\varepsilon k \}$.
Therefore, by plurifine locality, we obtain
\begin{equation}\label{eq: thm31}
	\begin{split}
	(1 - 2\varepsilon )^{ n } \int_{ X\backslash E_k }  \theta_{ \varphi_k } ^n   + O(1) \epsilon \int_{ X } \theta^n  \geq \int_{ \{ u_k \leq  \psi - \varepsilon k \} }  \theta_{ \psi }^n  \geq \int_{ \{ \psi \geq  - \varepsilon k \} }  \theta_{ \psi }^n 
	\end{split}
\end{equation} noticing that $\varphi_k\leq 0$. By plurifine locality again, we have
\[ \int_{ X\backslash E_k }  \theta_{ \varphi_k } ^n \leq \int_{\{\varphi>-k-1\}}\theta_\varphi^n+\int_{\{\varphi\leq \psi-\varepsilon k\}}\theta_{\varphi_k}^n.\]
Plugging this into~\eqref{eq: thm31}, we obtain
\[ \int_{\{\varphi>-k-1\}} \theta_\varphi^n+\int_{\{\varphi\leq \psi -\varepsilon k\}}\theta_{\varphi_k}^n+O(1)\varepsilon\int_X\theta^n\geq \int_{\{\psi\geq -\varepsilon k\}}\theta_\psi^n\]
 Since $\varphi\geq \psi+ O(1)$, letting $k\to+\infty$ and then $\varepsilon\to 0^+$, we  obtain 
$$\int_{ X }  \theta_{ \varphi }^n  \geq \int_{ X } \theta_{ \psi }^n$$ since the measures $\theta_\varphi^n$ and $\theta_\psi^n$ vanish on pluripolar sets.
The desired equality follows by reversing the roles of $\varphi$ and $\psi$.

\smallskip
\noindent {\bf Step 2:} We treat the general case: $\{\theta \} $ is merely big.


We choose a real number $t_0>0$ so large that $\theta + t_0 \omega \geq 0$. By step 1, we have
$$\int_{ X }  (t\omega + \theta_{ \varphi } )^n  = \int_{ X } (t\omega + \theta_{ \psi } )^n ,$$
for all $t\geq t_0$. 
Since non-pluripolar products are multilinear; see \cite[Proposition 1.4]{boucksom2010monge}), we see that $\int_{ X }  (t\omega + \theta_{ \varphi } )^n  \;\text{and}\; \int_{X}  (t\omega + \theta_{ \psi } )^n $ are both homogeneous polynomials of degree $n$. Therefore, we obtain an equality between two polynomials for all $t\geq 0$. Identifying the constant coefficients of these polynomials yields the desired equality.
 \end{proof}
 Next, we establish a slight generalization of Witt-Nystr\"om's result, which is used to prove Theorem~\ref{mainthm}.
\begin{theorem}\label{thm: sub-main} Let $\varphi,\psi\in\PSH(X,\theta)$. If $\varphi$ and $\psi$ have the same singularity type in capacity, i.e.,
    there exist $0\geq h\in\mathcal{E}(X,\omega)$ and $c\geq 0$ such that 
	\[\lim_{k\to+\infty}k^n\capK_{\omega}(\{|\varphi-\psi|>-ch+k\})=0,  \]
	then $\int_X\theta_\varphi^n=\int_X\theta_\psi^n$.
\end{theorem}
\begin{proof} According to Proposition~\ref{prop1}, it suffices to treat the case where $\theta$ is semi-positive.
We borrow the notation in the proof of Proposition~\ref {prop1}. By the same arguments (see~\eqref{eq: thm31}), we have
\begin{equation}\label{eq}
	\begin{split}
	(1 - 2\varepsilon )^{ n } \int_{ X\backslash E_k }  \theta_{ \varphi_k } ^n   + O(1) \epsilon \int_{ X } \theta^n \geq \int_{ \{ \psi \geq  - \varepsilon k \} }  \theta_{ \psi }^n.
	\end{split}
 \end{equation} 
Since $\theta\leq \omega$, we have
\begin{equation*}\label{eq: eq1}
\begin{split}
\int_{ X\backslash E_k } \theta_{ \varphi_k }^n&\leq\int_{ \big\{\varphi > - k - 1 \big\} }  \theta_{ \varphi } ^n + \int_{ \big\{ \varphi \leq \psi - \varepsilon k \big\} } \theta_{ \varphi_k }^n\\
&\leq \int_{X}\theta_{ \varphi }^n  + \int_{ \{ \varphi < \psi + ch - \frac { \varepsilon ( k - 1 ) } { 2 } \} } \theta_{ \varphi_k }^n + \int_{ \big\{ h \leq - \frac { \varepsilon k } { 2c } \big\} } \theta_{ \varphi_k }^n\\
&\leq \int_{ X }  \theta_{ \varphi }^n  +  ( k + 1 )^n \capK_{ \omega } \bigg( \bigg\{ \varphi < \psi + ch - \frac { \varepsilon ( k - 1 ) } { 2 } \bigg\} \bigg) + \int_{ \{ h \leq - \frac { \varepsilon k } { 2c } \} }  \theta_{ \varphi_k }^n.
	\end{split}
\end{equation*}
The last term is dominated by
\begin{equation*}\label{eq: eq2}
\begin{split}
\int_{ \big\{ h\leq - \frac {\varepsilon k } { 2c } \big\} } \theta_{ \varphi_k }^n & \leq \int_{ \big\{ h < - \frac { \varepsilon k } { 4c } + \frac { \varepsilon ( k - 1 ) } { 4c ( k + 1 ) } \varphi_k \big\} } \theta_{ \varphi_k }^n\\
  &\leq \frac { (4c ( k + 1 ) )^n } {  (k - 1 )^n \varepsilon^n } \int_{ \big\{ h < - \frac { \varepsilon k } { 4c } + \frac { \varepsilon ( k - 1 ) } { 4c ( k + 1 ) } \varphi_k\big \} } \bigg( \theta+\ddc {  \frac { \varepsilon ( k - 1 ) } { 4c ( k + 1 ) } \varphi_k }  \bigg)^n\\
&\leq \frac { ( 4c ( k + 1 ) )^n } { ( k - 1 )^n \varepsilon^n } \int_{ \big\{ h < - \frac { \varepsilon k } { 4c } + \frac { \varepsilon ( k - 1 ) } { 4c ( k + 1 ) } \varphi_k \big\} } \bigg( \omega+\ddc {  \frac { \varepsilon ( k - 1 ) } { 4c ( k + 1 ) } \varphi_k }  \bigg)^n\\
&\leq \frac { ( 4c ( k + 1 ) )^n } { ( k - 1 )^n \varepsilon^n } \int_{ \big \{ h < - \frac { \varepsilon k } { 4c } + \frac { \varepsilon ( k - 1 ) } { 4c ( k + 1 ) } \varphi_k \big\} } (\omega + dd^ch )^n\\
&\leq \frac { ( 4c ( k + 1 ) )^n } { ( k - 1 )^n \varepsilon^n } \int_{ \big\{ h < - \frac { \varepsilon k } { 4c } \big\} } (\omega + dd^ch )^n,
\end{split}
\end{equation*} where in the fourth inequality, we applied the comparison principle for the class $\mathcal E(X, \omega)$ (see \cite[Theorem 1.5]{guedj2007weighted} or \cite[Theorem 2.3]{dinew2009uniqueness}).
Combining these inequalities, we obtain
\begin{equation}\label{eq1}
\begin{split}
	\int_{ X\backslash E_k } \theta_{ \varphi_k }^n&\leq \int_{ X }  \theta_{ \varphi }^n  + ( k + 1 )^n \capK_{ \omega } \bigg( \bigg\{ \varphi < \psi + ch - \frac { \varepsilon ( k - 1 ) } { 2 } \bigg\} \bigg) \\
	&\quad +\frac { ( 4c ( k + 1 ) )^n } { ( k - 1 )^n \varepsilon^n } \int_{ \big\{ h < - \frac { \varepsilon k } { 4c } \big\} } (\omega + dd^ch )^n.
\end{split}
\end{equation}
Plugging \eqref{eq1} into \eqref{eq}, we obtain
\begin{align*}
&\int_{ X } \theta_{ \varphi }^n  + 
(k + 1 )^n \capK_{ \omega } \bigg( \bigg\{ \varphi < \psi + ch - \frac { \varepsilon ( k - 1 ) } { 2 } \bigg\} \bigg) 
 \\&\quad  
 +\frac { ( 4c ( k + 1 ) )^n } { ( k - 1 )^n \varepsilon^n } \int_{ \{ h < - \frac { \varepsilon k } { 4c } \} } (\omega + dd^ch )^n   + O(1) \varepsilon \int_{ X } \theta^n\\
&\qquad\geq \int_{ \{ \psi \geq  - \varepsilon k \} } \theta_{ \psi }^n.\end{align*}
Since the measures $ \theta_{\varphi}^n $, $\theta_{\psi}^n $, and $\omega_h^n$ vanish on pluripolar sets, letting $k\to +\infty$, and then $\epsilon\to 0^+$ we obtain
$$\int_{ X }  \theta_{ \varphi }^n  \geq \int_{ X } \theta_{ \psi }^n.$$
The desired equality follows by reversing the roles of $\varphi$ and $\psi$. \end{proof}




\begin{proof}[Proof of Theorem \ref{mainthm}]	
 Arguing the same as in the proof of Proposition~\ref{prop1}, we can assume that the classes $\{\theta_j\}$ are big.

 We first consider the case that $\varphi_j$ have the same singularity type in capacity as $\psi_j$ for $1\leq j\leq n$, i.e.,
$$\lim\limits_{ k \to +\infty }k^n \capK_{ \omega } ( \{ |\varphi_j - \psi_j | > -ch + k \} ) = 0,$$
 for some $c\geq 0$, and $ h\in \mathcal E(X,\omega)$. By Theorem \ref{thm: sub-main}, we have
$$\int_{ X }  \Big(\sum\limits_{ j = 1 }^{ n } t_j \theta_{ j, \varphi_j } \Big)^n  = \int_{ X }  \Big( \sum\limits_{ j = 1 }^{ n } t_j \theta_{ j, \psi_j } \Big)^n ,$$
for all $t_1, \ldots, t_n\geq 0$. Using the 
multilinearity of the non-pluripolar product (see \cite[Proposition 1.4]{boucksom2010monge}), we see that both $\int_{ X }  ( \sum_{ j = 1 }^{ n } t_j \theta_{ \varphi_j } )^n $ and $\int_{ X }  (\sum_{ j = 1 }^{ n } t_j \theta_{ \psi_j } )^n $ are homogeneous polynomials of degree $n$. Therefore, our last identity ensures that all the coefficients of these polynomials have to be equal, which implies
$$\int_X  \theta_{ 1, \varphi_1 }\wedge \ldots \wedge \theta_{ n, \varphi_n }  = \int_X  \theta_{ 1, \psi_1 }\wedge \ldots\wedge \theta_{ n, \psi_n }.$$
We now treat the general case. We set
$$u_{j\ell} = \max ( \varphi_j - \ell, \psi_j ).$$
We see that $\varphi_j$ and $u_{j\ell}$ have the same singularity type in capacity for $j=1,\ldots,n$. The previous case yields
$$\int_X  \theta_{ 1, \varphi_1 }\wedge ...\wedge \theta_{ n, \varphi_n } = \int_X  \theta_{ 1, u_{1\ell} } \wedge ...\wedge \theta_{ n, u_{n\ell} }  .$$
Letting $\ell\to +\infty$, by \cite[Theorem 2.3]{darvas2018monotonicity}
  (see \cite[Proposition 2.7]{Hiep10-CV-capacity} in the local setting), we obtain
\[\int_X  \theta_{1, \varphi_1 }\wedge \ldots  \wedge \theta_{ n, \varphi_n } \geq \int_X \theta_{ 1, \psi_1 }\wedge \ldots\wedge \theta_{ n, \psi_n } .\]\end{proof}

As a direct consequence, we have the following.
\begin{corollary}\label{thm: compare-mass} Let $\{\theta\}$ be a big cohomology class.
Assume $\varphi,\psi\in\PSH(X,\theta)$ such that $\varphi\leq \psi$. If $\int_X\theta_\varphi^n=\int_X\theta_\psi^n$ then $P_\omega(b(\varphi-\psi))\in\mathcal E(X,\omega)$ for all $b>0$.
 
		 Conversely,  if there exists a constant $b>0$ such that $P_\omega(b(\varphi-\psi))\in\mathcal E(X,\omega)$ then $\int_X\theta_\varphi^n=\int_X\theta_\psi^n$.	
\end{corollary}

\begin{proof}
Since $\int_X\theta_\varphi^n=\int_X\theta_\psi^n$, it follows that $\varphi\in\mathcal{E}(X,\theta,\psi)$. Consequently, Proposition~\ref{prop: homega} ensures that $P_\omega(b(\varphi-\psi))\in\mathcal{E}(X,\omega)$ for any $b>0$.

In contrast, setting $h=P_\omega(b(\varphi-\psi))$ for some $b>0$, we observe that
$\varphi\geq \psi +b^{-1}h$ on $X$.  Applying Theorem~\ref{mainthm}, we obtain the desired inequality. 
\end{proof}
Applying Theorem~\ref{mainthm}, we provide the proof of Theorem~\ref{thm: thm1}.
\begin{proof}[Proof of Theorem~\ref{thm: thm1}]
Set $b\colonequals 2c^{-1}$. Since $\varphi,\psi\in\mathcal{E}(X,\theta,\phi)$, by Proposition~\ref{prop: homega}, both functions $u_b\colonequals P_\omega(b(\varphi-\psi))$ and $v_b\colonequals P_\omega(b(\psi-\varphi))$ belong to the class $\mathcal{E}(X,\omega)$. We set
$h_b\colonequals\frac{u_b+v_b-C_b}{2}$, where $C_b>0$ is an upper bound for $u_b$ and $v_b$. We observe that
\[ch_b=b^{-1}(u_b+v_b-C_b)\leq \varphi-\psi \] since $v_b-C_b\leq 0$.  Similarly, $ch_b\leq \psi-\varphi$, and the first statement follows.   

The second statement follows immediately from Theorem~\ref{mainthm}, while the last one is the special case where $\psi=\phi$.  
\end{proof}

\begin{example}
We provide here an example to show that in Theorem~\ref{thm: thm1}, the bigness of $\theta$ is necessary. Let $X=X_1\times X_2$ where $(X_1,\omega_1)$, $(X_2,\omega_2)$ are two compact K\"ahler manifolds of dimension $p\geq 1$, $q\geq 1$ respectively, with $n=\dim X=p+q$.  
Define a K\"ahler form on $X$ by $\omega=\pi_1^*\omega_1+\pi_2^*\omega_2$ is , where $\pi_j$ is the projection onto the $j$-th factor.
The smooth closed (1,1) form $\theta$ is defined by $\pi_1^*\omega_1$. We easily see that $V_\theta=0$ and $\{\theta\}$ is not big since $\Vol(\theta)=\int_X\theta^n=0$. 
For any $\varphi\in\PSH(X_1,\omega_1)$, we can find an extension $\Tilde{\varphi}$  of $\varphi$ which belongs to $\PSH(X,\theta)$. 
We have $\int_X\theta_{\Tilde{\varphi}}^n=0$, hence $\Tilde{\varphi}\in\mathcal{E}(X,\theta)$. If Theorem~\ref{thm: thm1} holds, it follows that $\nu(\Tilde\varphi,x)=\nu(V_\theta,x)=0$ for all $x\in X$. This leads to a contradiction since  $\varphi\in\PSH(X_1,\omega_1)$ can be chosen to have a positive Lelong number at some point.
\end{example}

\begin{corollary}
 Let $\{\eta\}$ and $\{\theta\}$ be big cohomology classes. Let $\phi\in\PSH(X,\theta)$. Then the following are equivalent.
\hfill
\begin{enumerate}
\item $\PSH(X,\eta)\cap\mathcal{E}(X,\theta,\phi)\neq \varnothing$.

\item There exists $u\in\PSH(X,\eta)$, $v\in\mathcal{E}(X,\theta,\phi)$, $c>0$, and $h\in\mathcal{E}(X,\omega)$ such that \[u\geq v+ ch. \]

\item $\mathcal{E}(X,\eta)\cap\PSH(X,\theta,\phi)\subset \mathcal{E}(X,\theta,\phi)$.
\end{enumerate}

\end{corollary}

\begin{proof}
 We see that $(1)\Rightarrow (2)$ is trivial. Indeed, we can take $u=v\in \PSH(X,\eta)\cap\mathcal{E}(X,\theta,\phi)$ and $h=0$.

For $(2)\Rightarrow (3)$, let $\varphi\in \mathcal{E}(X,\eta)\cap \PSH(X,\theta,\phi)$. Thanks to Proposition \ref{prop: homega}, we have $h'\coloneqq P_\omega(\varphi-V_\eta)\in\mathcal{E}(X,\omega)$. By assumption, there exists $u\in\PSH(X,\eta)$, $v\in\mathcal{E}(X,\theta,\phi)$, $c>0$, and $h\in\mathcal{E}(X,\omega)$ such that \[u\geq v+ ch. \]
 It follows that 
 \[\varphi\geq V_\eta+h'\geq u-\sup_X u+h'\geq v+ch+h'-\sup_X u. \]
 Theorem \ref{thm: thm1} yields $\int_X\theta_\varphi^n=\int_X\theta^n_v=\int_X\theta_\phi^n$, hence $\varphi\in\mathcal{E}(X,\theta,\phi)$.

 For $(3)\Rightarrow (1)$, if $\mathcal{E}(X,\eta)\cap\PSH(X,\theta,\phi)=\varnothing$, we are done. Otherwise, $u\in \mathcal{E}(X,\eta)\cap\PSH(X,\theta,\phi)$, so $u\in\mathcal{E}(X,\theta,\phi)$. 
\end{proof}
\begin{remark}
When $\eta\geq \theta$, we obviously have $\PSH(X,\eta)\cap\mathcal{E}(X,\theta,\phi)=\mathcal{E}(X,\theta,\phi)\neq \varnothing$. 
\end{remark}

\bibliographystyle{alpha}
\bibliography{bibfile}

\begin{thebibliography}{{\AA}CCP09}

\bibitem[{\AA}CCP09]{Ahag09-MA-pluripolar}
P.~{\AA}hag, U.~Cegrell, R.~Czy{\.z}, and H.~H. Pha{\d{m}}.
\newblock Monge-{Amp{\`e}re} measures on pluripolar sets.
\newblock {\em J. Math. Pures Appl. (9)}, 92(6):613--627, 2009.

\bibitem[BEGZ10]{boucksom2010monge}
S.~Boucksom, P.~Eyssidieux, V.~Guedj, and A.~Zeriahi.
\newblock Monge-{A}mp\`ere equations in big cohomology classes.
\newblock {\em Acta Math.}, 205(2):199--262, 2010.

\bibitem[Bou02]{boucksom2002volume}
S.~Boucksom.
\newblock On the volume of a line bundle.
\newblock {\em Internat. J. Math.}, 13(10):1043--1063, 2002.

\bibitem[Bou04]{boucksom2004divisorial}
S.~Boucksom.
\newblock Divisorial {Z}ariski decompositions on compact complex manifolds.
\newblock {\em Ann. Sci. \'{E}cole Norm. Sup. (4)}, 37(1):45--76, 2004.

\bibitem[BT76]{bedford1976dirichlet}
E.~Bedford and B.~A. Taylor.
\newblock The {D}irichlet problem for a complex {M}onge-{A}mp\`ere equation.
\newblock {\em Invent. Math.}, 37(1):1--44, 1976.

\bibitem[BT82]{bedford1982new}
E.~Bedford and B.~A. Taylor.
\newblock A new capacity for plurisubharmonic functions.
\newblock {\em Acta Math.}, 149(1-2):1--40, 1982.

\bibitem[BT87]{bedford1987fine}
E.~Bedford and B.~A. Taylor.
\newblock Fine topology, \v{S}ilov boundary, and {$(dd^c)^n$}.
\newblock {\em J. Funct. Anal.}, 72(2):225--251, 1987.

\bibitem[Ceg98]{cegrell1998pluricomplex}
U.~Cegrell.
\newblock Pluricomplex energy.
\newblock {\em Acta Math.}, 180(2):187--217, 1998.

\bibitem[Ceg04]{cegrell2004general}
U.~Cegrell.
\newblock The general definition of the complex {M}onge-{A}mp\`ere operator.
\newblock {\em Ann. Inst. Fourier (Grenoble)}, 54(1):159--179, 2004.

\bibitem[Dan22]{dang2021continuity}
Q.-T. Dang.
\newblock Continuity of {M}onge-{A}mp\`ere potentials in big cohomology classes.
\newblock {\em Int. Math. Res. Not. IMRN}, (14):11180--11201, 2022.

\bibitem[Dan23]{dang2023continuity}
Q.-T. Dang.
\newblock Continuity of {M}onge-{A}mp\`ere potentials with prescribed singularities.
\newblock {\em J. Geom. Anal.}, 33(10):Paper No. 318, 15, 2023.

\bibitem[DDL18a]{darvas2018monotonicity}
T.~Darvas, E.~Di{ }Nezza, and C.~H. Lu.
\newblock Monotonicity of nonpluripolar products and complex {M}onge-{A}mp\`ere equations with prescribed singularity.
\newblock {\em Anal. PDE}, 11(8):2049--2087, 2018.

\bibitem[DDL18b]{darvas2018singularity}
T.~Darvas, E.~Di{ }Nezza, and C.~H. Lu.
\newblock On the singularity type of full mass currents in big cohomology classes.
\newblock {\em Compos. Math.}, 154(2):380--409, 2018.

\bibitem[DDL21a]{darvas2021log}
T.~Darvas, E.~Di{ }Nezza, and C.~H. Lu.
\newblock Log-concavity of volume and complex {M}onge-{A}mp\`ere equations with prescribed singularity.
\newblock {\em Math. Ann.}, 379(1-2):95--132, 2021.

\bibitem[DDL21b]{darvas2020metric}
T.~Darvas, E.~Di{ }Nezza, and C.~H. Lu.
\newblock The metric geometry of singularity types.
\newblock {\em J. Reine Angew. Math.}, 771:137--170, 2021.

\bibitem[DDL25]{darvas2023relative}
T.~Darvas, E.~Di{ }Nezza, and C.~H Lu.
\newblock Relative pluripotential theory on compact {K{\"a}hler} manifolds.
\newblock {\em Pure Appl. Math. Q.}, 21(3):1037--1118, 2025.

\bibitem[Dem93]{demailly1993monge}
J.-P. Demailly.
\newblock Monge-{A}mp\`ere operators, {L}elong numbers and intersection theory.
\newblock In {\em Complex analysis and geometry}, Univ. Ser. Math., pages 115--193. Plenum, New York, 1993.

\bibitem[DGZ16]{dinew2016open}
S.~Dinew, V.~Guedj, and A.~Zeriahi.
\newblock Open problems in pluripotential theory.
\newblock {\em Complex Var. Elliptic Equ.}, 61(7):902--930, 2016.

\bibitem[Din09]{dinew2009uniqueness}
S.~Dinew.
\newblock Uniqueness in {$\mathcal{E}(X,\omega)$}.
\newblock {\em J. Funct. Anal.}, 256(7):2113--2122, 2009.

\bibitem[GZ07]{guedj2007weighted}
V.~Guedj and A.~Zeriahi.
\newblock The weighted {M}onge-{A}mp\`ere energy of quasiplurisubharmonic functions.
\newblock {\em J. Funct. Anal.}, 250(2):442--482, 2007.

\bibitem[Hie10]{Hiep10-CV-capacity}
P.~H. Hiep.
\newblock Convergence in capacity and applications.
\newblock {\em Math. Scand.}, 107(1):90--102, 2010.

\bibitem[Ko{\l}03]{kolodziej2003monge}
S.~Ko{\l}odziej.
\newblock The {M}onge-{A}mp\`ere equation on compact {K}\"{a}hler manifolds.
\newblock {\em Indiana Univ. Math. J.}, 52(3):667--686, 2003.

\bibitem[LN22]{lu2022hessian}
C.~H. Lu and V.-D. Nguy\^en.
\newblock Complex {H}essian equations with prescribed singularity on compact {K}\"ahler manifolds.
\newblock {\em Ann. Sc. Norm. Super. Pisa Cl. Sci. (5)}, 23(1):425--462, 2022.

\bibitem[Lu21]{lu2021capacities}
C.~H. Lu.
\newblock Comparison of {M}onge-{A}mp\`ere capacities.
\newblock {\em Ann. Polon. Math.}, 126(1):31--53, 2021.

\bibitem[RW14]{ross2014analytic}
J.~Ross and D.~Witt{ }Nystr\"{o}m.
\newblock Analytic test configurations and geodesic rays.
\newblock {\em J. Symplectic Geom.}, 12(1):125--169, 2014.

\bibitem[Vu21]{vu2021relative}
D.-V. Vu.
\newblock Relative non-pluripolar product of currents.
\newblock {\em Ann. Global Anal. Geom.}, 60(2):269--311, 2021.

\bibitem[WN19]{wittnystrom19-monotonicity}
D.~Witt~Nystr{\"o}m.
\newblock Monotonicity of non-pluripolar {Monge}-{Amp{\`e}re} masses.
\newblock {\em Indiana Univ. Math. J.}, 68(2):579--591, 2019.

\end{thebibliography}

\end{document}